\documentclass[12pt,a4paper,oneside,reqno,notitlepage]{amsart}
\usepackage{amssymb}

\usepackage{amsfonts}
\usepackage{amscd,amsmath}
\usepackage{euscript}
\usepackage{comment}

\usepackage[arrow,matrix,curve]{xy}\SilentMatrices
\def\xyma{\xymatrix@M.7em}

\numberwithin{equation}{section}

\usepackage[T2A]{fontenc}
\newtheorem{cor}{Corollary}[section]

\newtheorem{prop}{Proposition}[section]
\newtheorem{theorem}{Theorem}[section]
\newtheorem{lemma}{Lemma}[section]
\newtheorem{remark}{Remark}[section]

\def\bee{\begin{equation}}
\def\ee{\end{equation}}

\def\lotimes{\buildrel{L}\over\otimes}
\begin{document}
\title{On the splitting of polynomial functors}
\author{Roman Mikhailov}

\address{Steklov Mathematical Institute, Gubkina 8, 119991 Moscow, Russia and Institute for Advanced Study, Princeton, NJ, USA}
\email{romanvm@mi.ras.ru}
\urladdr{http://www.mi.ras.ru/\~{}romanvm/pub.html}

\begin{abstract}
We develop methods for proving that certain extensions of
polynomial functors do not split naturally. As an application we
give a functorial description of the third and the fourth stable
homotopy groups of the classifying spaces of free abelian groups.
\end{abstract}

\thanks{This material is based upon work
supported by the National Science Foundation under agreement No.
DMS-0635607. Any opinions, findings and conclusions or
recommendations expressed in this material are those of the
authors and do not necessarily reflect the views of the National
Science Foundation.} \maketitle

\section{Introduction}
\vspace{.5cm}

There are different important spaces whose homotopy type depends
naturally on an abelian group $A$. For example, Eilenberg-MacLane
spaces $K(A,n),\ n\geq 1$, suspensions  $\Sigma^mK(A,n)$ of these,
$m\geq 1$, wedges $K(A,n)\vee K(A,m)$, etc... The homology and
homotopy groups of these spaces can be viewed as functors on the
category of abelian groups. The problem of describing these as
functors is more difficult than an abstract description. Indeed,
an abstract description of homology groups $H_n(A)$ of abelian
groups is simple, as it follows from the K\"unneth formula,
whereas the functors $A\mapsto H_n(A)$ are complicated
\cite{Breen}. This paper continues the research started in
\cite{Breen}, \cite{BreenMikhailov}. It is standard that the
functorial description of different homological or homotopical
functors follows from certain spectral sequences. As a rule, the
result of the convergence of a spectral sequence gives not a
functor as a whole, but only a filtration on it. To solve the
extension problem, i.e. to glue the real functor from different
pieces, one needs methods for the control of functorial
extensions. Such methods are developed in this paper. Note that
all the functors considered in this paper are over $\mathbb Z$,
they are defined on the category of (free) abelian groups; observe
that the analogous results over fields can be obtained more
easily.

To give some examples which illustrate the spirit of questions
considered in this paper, let us start with two complexes of
abelian groups $C_*, D_*$, one can compute the homology of their
tensor product $H(C_*\otimes D_*)$ in terms of $H(C_*)$ and
$H(D_*)$ using the well-known K\"unneth formula. Now consider
three abelian groups $A,B,C$. The K\"unneth formula gives the
following exact sequence
\begin{equation}\label{maclane}
0\to Tor(A,B)\otimes C\to H_1\left(A\lotimes B\lotimes C\right)\to
Tor(A\otimes B,C)\to 0
\end{equation}
which splits as a sequence of abelian groups. The middle term of
this sequence is the functor $Trip(A,B,C)$ of MacLane
\cite{MacLane} which is simply the first homology group of the
iterated tensor product in the derived category. S. MacLane proved
in \cite{MacLane} that the sequence (\ref{maclane}) does not split
naturally as a sequence of multi-functors. On the other hand, let
us fix the two groups $B=C=\mathbb Z/2$. In this case the sequence
(\ref{maclane}) has the form
$$
0\to Tor(A,\mathbb Z/2) \to H_1\left(A\lotimes \mathbb Z/2\lotimes
\mathbb Z/2\right)\to  A\otimes \mathbb Z/2\to 0,
$$
and this sequence splits as a sequence of functors. This simply
follows from the fact that we can choose a splitting $A\lotimes
\mathbb Z/2\lotimes \mathbb Z/2\simeq A\lotimes(\mathbb Z/2\oplus
\mathbb Z/2[1])$ in the derived category functorially in
$A$\footnote{We will always use the traditional notation $[n]$ for
the shift of degree $n$ in the derived category.}.

One more example is the following. Let $A$ be an abelian group. A
description of the third homology of $A$ as a functor is given in
\cite{Breen}. There is the following natural exact sequence
\begin{equation}\label{seqh3}
0\to \Lambda^3(A)\to H_3(A)\to \Omega_2(A)\to 0
\end{equation}
where $\Lambda^3$ is the third exterior power and $\Omega_2$ is
the quadratic functor due to Eilenberg-MacLane, which is in fact
the first derived functor of the exterior square. The sequence
(\ref{seqh3}) splits as a sequence of abelian groups. One can ask
weather (\ref{seqh3}) splits as a sequence of functors. We prove
(see corollary \ref{corr2}) that (\ref{seqh3}) this is not the
case. More generally, we will prove that, for all $n\geq 3$, the
natural injection $\Lambda^n(A)\hookrightarrow H_n(A)$ induced by
Pontryagin product in homology, does not split naturally (see
proposition \ref{teq})\footnote{After posting the paper the author
was informed by N.Kuhn that this result follows from example 7.6
\cite{Kuhn}.}.

The paper is organized as follows. We recall in section 2 the
description of the polynomial functors on the category of free
abelian groups in terms of maps between cross-effects from
\cite{BDFP}. The language of polynomial $\mathbb Z$-modules
developed in \cite{BDFP} and \cite{Bau} is useful for the
description of $Hom$ and $Ext$-groups for polynomial functors in
the category of free abelian groups. We use this language for
proving that certain exact sequences do not split.

We describe the third stable homotopy group of $K(A,1)$ as a
functor in section \ref{thirdsec}.  We show that, for a free
abelian group $A$, there is a short exact sequence
$$
0\to S^2(A)\otimes \mathbb Z/2\to \pi_3^SK(A,1)\to \Lambda^3(A)\to
0
$$
which does not split naturally. Moreover, the functor
$\pi_3^SK(A,1)$ represents the unique non-trivial element in the
group of functorial extensions $Ext(\Lambda^3,S^2\otimes\mathbb
Z/2)=\mathbb Z/2$. In section \ref{fourthsec} we give a functorial
description of the fourth stable homotopy group of $K(A,1)$ for a
free abelian group $A$.

On the other hand, there are some cases in which the existence of
a functorial splitting can be proved. For an object $C$ of the
derived category of abelian groups concentrated in non-positive
dimensions $\sf DAb_{\leq 0}$, we show that the exact triangle in
$\sf DAb_{\leq 0}$
$$
LS^2(C[1])\to L\Gamma_2(C[1])\to C\lotimes \mathbb Z/2[1]\to
LS^2(C[1])[1]
$$
induces the functorial splitting (theorem \ref{the1})
\begin{equation}\label{mz7}
\pi_i(L\Gamma_2(C[1]))\simeq \pi_i(LS^2(C[1])))\oplus
\pi_i\left(C\lotimes \mathbb Z/2[1]\right),\ i\geq 1
\end{equation}
While the well-known theorem of Dold \cite{Dold} implies that
there is a splitting on the level of complexes which induces the
splitting (\ref{mz7}) on homotopy, we show that this splitting is
not functorial, i.e.
$$
L\Gamma_2(C[1])\neq LS^2(C[1])\oplus C\lotimes\mathbb Z/2[1]
$$
in the derived category $\sf DAb_{\leq 0}$.

\vspace{.5cm}
\section{Polynomial functors}
\vspace{.5cm} Denote by $\sf Ab$ (resp. $\sf fAb$) the category of
finitely generated abelian (resp. f.g. free abelian) groups. For a
small category $\sf C$, let $Fun(\sf C, \sf Ab)$ be the category
of functors from $\sf C$ to $\sf Ab$. Morphisms in $Fun(\sf C,\sf
Ab)$ are natural transformations between functors. It is
well-known that $Fun(\sf C,\sf Ab)$ is an abelian category with
enough projectives and injectives. By $\sf DAb_{\leq 0}$ we mean
the derived category of abelian groups living in non-negative
degrees which is equivalent to the homotopy category of simplicial
abelian groups via the Dold-Kan correspondence \cite{DoldPuppe}.

The main functors which we will consider are the following $(n\geq
1)$:
\begin{itemize}
\item Tensor powers $\otimes^n: \sf Ab\to Ab$\\

\item Symmetric powers $S^n: \sf Ab\to Ab$\\

\item Exterior powers $\Lambda^n: \sf Ab\to Ab$\\

\item Divided powers $\Gamma_n: \sf Ab\to Ab$\\

\item Antisymmetric square $\widetilde \otimes^2: \sf Ab\to Ab$,
defined as
$$
\widetilde\otimes^2(A):=A\otimes A/\{a\otimes b+b\otimes a,\
a,b\in A\}
$$
\end{itemize}

We will use the same notation for functors on $\sf Ab$ and for
their restriction on $\sf fAb$. To distinguish $Hom$ and $Ext$
groups for functors on $\sf Ab$ and $\sf fAb$, we will use the
notation $Hom(F,G)$ (resp. $Ext(F,G)$) for ordinal natural
transformations (resp. extensions) of functors $F,G: \sf Ab\to Ab$
and $Hom_{f}(F,G)$ (resp. $Ext_f(F,G)$) for functors $F,G: \sf
fAb\to Ab$.

Let $F: {\sf Ab\to Ab}$ be a functor. Recall that the {\it
cross-effects} of $F$ are multi-functors defined as
\begin{equation}\label{defcross}
F(X_1|\dots |X_n)=ker\{F(X_1\oplus \dots \oplus X_n)\to
\oplus_{i=1}^n F(X_1\oplus \dots \hat X_i\dots \oplus X_n )\},\
X_i\in {\sf Ab},\ n\geq 2
\end{equation}
where the maps $F(X_1\oplus \dots \oplus X_n)\to F(X_1\oplus \dots
\hat X_i\dots \oplus X_n)$ are induced by natural retractions. The
functor $F$ is polynomial of degree $d$ $(d\geq 1)$ if $F(0)=0$
and $F(X_1|\dots|X_d)$ is linear in each variable $X_i,\
i=1,\dots, d$.

Given a functor $F$ and an abelian group $A$, consider the system
of abelian groups:
$$
F_1=F(A),\ F_2=F(A|A),\ \dots, F_n=F(A|\dots |A)\ (n\ \text{copies
of}\ A)
$$
together with the homomorphisms
$$
H_m^n: F_n(A)\to F_{n+1}(A),\ P_m^n: F_{n+1}(A)\to F_n(A),
m=1,2\dots,\ m<n
$$
which are defined as composite maps
\begin{align*}
& H_m^{n+1}: F_n(A)\hookrightarrow F(A^{\oplus n})\to
F(A^{\oplus n+1})\twoheadrightarrow F_{n+1}(A)\\
& P_m^{n+1}: F_{n+1}(A)\to F(A^{\oplus n+1})\to F(A^{\oplus n})\to
F_n(A),
\end{align*}
induced by natural maps $A^{\oplus n}\to A^{\oplus n+1},\
A^{\oplus n+1}\to A^{\oplus n}$ given by
\begin{align*}
& (a_1,\dots, a_n)\to (a_1,\dots, a_m,a_m,\dots, a_n)\\
& (a_1,\dots, a_{n+1})\to (a_1,\dots, a_{m-1},
a_m+a_{m+1},a_{m+2},\dots, a_{n+1}).
\end{align*}
Denote the system of these abelian groups and maps by $J_F(A)$:
$$
J_F(A):\ F_1(A)\
\begin{matrix}\longrightarrow\\[-2.5mm]
\longleftarrow\end{matrix}\ F_2(A)\
\begin{matrix}\longrightarrow\\[-3.5mm]\longrightarrow\\[-3.5mm]
\longleftarrow\\[-3.5mm]\longleftarrow \end{matrix}\
F_3(A)\
\begin{matrix}\longrightarrow\\[-2.5mm]\dots\\[-2.5mm]\longleftarrow
\end{matrix}\ \cdots
$$

These maps satisfy certain standard relations \cite{BDFP}, which
do not depend on $F$ and $A$. For a polynomial functor $F$ of
degree $d$ and an abstract collection of $d$ abelian groups
$\{F_i(\mathbb Z)\}_{i=1,\dots,d}$ together with corresponding
maps which satisfy these relations is known as $d$-polynomial
$\mathbb Z$-module. Polynomial functors from free abelian groups
to abelian groups can be described in terms of polynomial $\mathbb
Z$-modules \cite{BDFP}. We now consider some particular cases.

\vspace{.5cm}\noindent {\it 1. Quadratic functors.} In the case of
quadratic functors (see \cite{Bau}), the required relations are
simple:
\begin{equation}\label{qzm}
A_1\buildrel{H_1^2}\over\longrightarrow\ A_2\
\buildrel{P_1^2}\over\longrightarrow\ A_1
\end{equation}
$$
H_1^2P_1^2H_1^2=2H_1^2,\ \ P_1^2H_1^2P_1^2=2P_1^2
$$
Such a diagram of abelian groups is called a quadratic $\mathbb
Z$-module. It is easy to compute the quadratic $\mathbb
Z$-modules, which correspond to the classical quadratic functors
mentioned above and to $-\otimes \mathbb Z/2$. Here they are:
\begin{align*}
& \mathbb Z^{\otimes^2}=(\mathbb Z\buildrel{(1,1)}\over\rightarrow
\mathbb Z\oplus\mathbb Z
\buildrel{(1,1)}\over\rightarrow \mathbb Z)\\
& \mathbb Z^{\Lambda^2}=(0\rightarrow \mathbb Z\rightarrow 0)\\
& \mathbb Z^{\Gamma_2}=(\mathbb Z\buildrel{1}\over\rightarrow
\mathbb
Z \buildrel{2}\over\rightarrow \mathbb Z)\\
& \mathbb Z^{SP^2}=(\mathbb Z\buildrel{2}\over\rightarrow \mathbb
Z
\buildrel{1}\over\rightarrow \mathbb Z)\\
& \mathbb Z^{\tilde \otimes^2}=(\mathbb Z/2\buildrel{0}\over\rightarrow \mathbb Z\buildrel{1}\over\rightarrow \mathbb Z/2)\\
 & \mathbb Z^{\mathbb Z/2}=(\mathbb
Z/2\rightarrow 0\rightarrow \mathbb Z/2)
\end{align*}
Given a quadratic $\mathbb Z$-module $M$ (\ref{qzm}), one can
define a quadratic functor on the category of abelian groups as
follows (see 6.13 \cite{Bau}): for an abelian group $A$, $A\otimes
M$ is generated by the symbols $a\otimes m, \{a,b\}\otimes n,\
a,b\in A,\ m\in A_1, n\in A_2$ with the relations
\begin{align*}
& (a+b)\otimes m=a\otimes m+b\otimes m+\{a,b\}\otimes H_1^2(m),\\
& \{a,a\}\otimes n=a\otimes P_1^2(n),\\
& a\otimes m\ \text{is linear in}\ m,\\
& \{a,b\}\otimes n\ \text{is linear in}\ a,b, n.
\end{align*}
The correspondence $A\mapsto A\otimes M$ defines a quadratic
functor and an equivalence between categories of quadratic
$\mathbb Z$-modules and quadratic functors $\sf fAb\to Ab$.

\vspace{.5cm}\noindent {\it 2. Cubical functors.} The cubical
$\mathbb Z$-module is given by the diagram
$$ A_1\
\begin{matrix}\buildrel{H_1^2}\over\longrightarrow\\[-2.5mm]
\buildrel{P_1^2}\over\longleftarrow\end{matrix}\ A_2\
\begin{matrix}\buildrel{H_2^3}\over\longrightarrow\\[-2.5mm]\buildrel{H_1^3}\over\longrightarrow\\[-2.5mm]
\buildrel{P_2^3}\over\longleftarrow\\[-2.5mm]\buildrel{P_1^3}\over\longleftarrow \end{matrix}
\ A_3
$$
with the following relations (see \cite{BDFP}, \cite{Drozd}):
\begin{align*}
& H_1^3H_1^2=H_2^3H_1^2,\ P_1^2P_1^3=P_1^2P_2^3,\ H_2^3P_1^3=0,\\
& H_1^3P_2^3=0,\ H_1^3P_1^3H_1^3=2H_1^3,\
P_1^3H_1^3P_1^3=2P_1^3,\\
& H_2^3P_2^3H_2^3=2H_2^3,\ P_2^3H_2^3P_2^3=2P_2^3,\\
& H_1^2P_1^2H_1^2=2H_1^2+2(P_1^3+P_2^3)H_1^3H_1^2,\\
& P_1^2H_1^2P_1^2=2P_1^2+2P_1^2P_1^3(H_2^3+H_1^3),\\
&
H_1^3H_1^2P_1^2+H_1^3+H_2^3=H_2^3P_2^3H_1^3P_1^3H_2^3+H_1^3P_1^3H_2^3P_2^3H_1^3,\\
&
H_1^2P_1^2P_1^3+P_1^3+P_2^3=P_2^3H_1^3P_1^3H_2^3P_2^3+P_1^3H_2^3P_2^3H_1^3P_1^3
\end{align*}
The simplest examples of the cubical $\mathbb Z$-modules, which
correspond to the exterior and symmetric cubes are the following:
\begin{align}
& \Lambda^3 \rightsquigarrow\ \ \ \ \ \ \ 0\ \begin{matrix}\longrightarrow\\[-2.5mm]
\longleftarrow\end{matrix}\ 0\
\begin{matrix}\longrightarrow\\[-2.5mm]\longrightarrow\\[-2.5mm]
\longleftarrow\\[-2.5mm]\longleftarrow \end{matrix}
\ \mathbb Z\label{l3m}\\ & S^3 \rightsquigarrow\ \ \ \ \ \ \
\mathbb Z\
\begin{matrix}\buildrel{(3,3)}\over\longrightarrow\\[-2.5mm]
\buildrel{(1,1)}\over\longleftarrow\end{matrix}\ \mathbb Z\oplus
\mathbb Z\
\begin{matrix}\buildrel{(0,2)}\over\longrightarrow\\[-2.5mm]\buildrel{(2,0)}\over\longrightarrow\\[-2.5mm]
\buildrel{(0,1)}\over\longleftarrow\\[-2.5mm]\buildrel{(1,0)}\over\longleftarrow \end{matrix}
\ \mathbb Z\label{sp3m}
\end{align}

\vspace{.5cm}\noindent 3. {\it $\Delta$-properties.} For any
functor $F$, the sequence
$$
\mathcal F(\mathbb Z):\ \ \ F_1(\mathbb Z)\
\buildrel{P_1^2}\over\longleftarrow\ F_2(\mathbb Z)\
\begin{matrix}\buildrel{P_2^3}\over\longleftarrow\\[-2.5mm]\buildrel{P_1^3}\over\longleftarrow \end{matrix}\
F_3(\mathbb Z)\
\begin{matrix}\longleftarrow\\[-2.5mm]\dots\\[-2.5mm]\longleftarrow
\end{matrix}\ \cdots
$$
is a $\Delta$-group, that is, the standard simplicial relations
for the face maps are satisfied. Taking the homology of this
complex, we obtain the values of the derived functors in the sense
of Dold-Puppe \cite{DoldPuppe}:
$$
H_i(\mathcal F(\mathbb Z))\simeq L_{i+1}F(\mathbb Z,1)
$$
This follows from the fact that the cross-effect spectral sequence
from \cite{DoldPuppe} degenerates to the complex $\mathcal
F(\mathbb Z)$.

\subsection{Natural transformations between functors}

All natural transformations between quadratic functors $\sf fAb\to
Ab$ are given as morphisms of corresponding quadratic $\mathbb
Z$-modules. One can therefor use quadratic $\mathbb Z$-modules for
the computation of the group of natural transformations between
given quadratic
functors. \\ \\
{\bf Examples.} 1. A natural map $S^2(A)\to \Gamma_2(A)$ is given
by the following diagram\footnote{For a map between cyclic groups
$f: A\to B$, we will use the notation $A\buildrel{n}\over\to B$ if
$f(a)=nb,$ where $a$ and $b$ are some given generators of $A$ and
$B$. Analogously we describe the maps between finitely-generated
abelian groups by integral matrices.}:
$$
\xyma{\mathbb Z \ar@{->}[r]^2 \ar@{->}[d]^{2k} & \mathbb
Z\ar@{->}[r]^1 \ar@{->}[d]^k &
\mathbb Z\ar@{->}[d]^{2k}\\
\mathbb Z\ar@{->}[r]^1 & \mathbb Z\ar@{->}[r]^2 &\mathbb Z}
$$
for $k\in\mathbb Z$, and $Hom_f(S^2,\Gamma_2)=\mathbb Z$.

\vspace{.5cm}\noindent 2. The natural map $\Gamma_2(A)\to A\otimes
\mathbb Z/2$ is given by the following diagram:
$$
\xyma{\mathbb Z \ar@{->}[r]^1 \ar@{->>}[d] & \mathbb
Z\ar@{->}[r]^2 \ar@{->}[d] &
\mathbb Z\ar@{->>}[d]\\
\mathbb Z/2\ar@{->}[r]& 0\ar@{->}[r] &\mathbb Z/2}
$$

\noindent 3. Let us now prove that there do not exist non-zero
natural maps
\begin{equation}\label{mz1}
\Lambda^2(A)\to S^2(A),\ \ S^2(A)\to A\otimes \mathbb Z/2
\end{equation}
Since the functors in (\ref{mz1}) are right exact, it is enough to
consider these functors on the category of free abelian groups.
Hence we can look at morphisms between the corresponding quadratic
$\mathbb Z$-modules. To every map $\Lambda^2\to SP^2$ corresponds
a morphism of quadratic $\mathbb Z$-modules:
$$
\xyma{0 \ar@{->}[r] \ar@{->}[d] & \mathbb Z\ar@{->}[r] \ar@{->}[d]
&
0\ar@{->}[d]\\
\mathbb Z\ar@{->}[r]^2 & \mathbb Z\ar@{->}[r]^1 &\mathbb Z}
$$
We see that the middle vertical map must be zero, hence the
result. Same reasoning applies to the natural transformation
$S^2(A)\to A\otimes \mathbb Z/2$:
$$
\xyma{\mathbb Z \ar@{->}[r]^2 \ar@{->}[d] & \mathbb Z\ar@{->}[r]^1
\ar@{->}[d] &
\mathbb Z\ar@{->}[d]\\
\mathbb Z/2\ar@{->}[r] & 0\ar@{->}[r] &\mathbb Z/2}
$$
We see that any such vertical map is zero. Hence, there is no any
non-zero natural transformations (\ref{mz1}).

\vspace{.5cm}\noindent 4. Consider the case of cubical functors.
The functors $S^3$ and $\Lambda^3$ are right exact, so that in
order to prove that $Hom(S^3,\Lambda^3)=0$ it is enough to show
that $Hom_f(S^3,\Lambda^3)=0$, i.e. that any map between cubical
$\mathbb Z$-modules \ref{sp3m} and \ref{l3m} is zero. It is easy
to see that all vertical maps in the following commutative
diagrams are zero:
\begin{center}
\begin{tabular}{cccccccccccc}
$\mathbb Z$ &
$\begin{matrix}\buildrel{(3,3)}\over\longrightarrow\\[-2.5mm]
\buildrel{(1,1)}\over\longleftarrow\end{matrix}$ & $\mathbb
Z\oplus \mathbb Z$ & $
\begin{matrix}\buildrel{(0,2)}\over\longrightarrow\\[-2.5mm]\buildrel{(2,0)}\over\longrightarrow\\[-2.5mm]
\buildrel{(0,1)}\over\longleftarrow\\[-2.5mm]\buildrel{(1,0)}\over\longleftarrow \end{matrix}$
& $\mathbb Z$\\
$\downarrow$ & & $\downarrow$ & & $ \downarrow$\\ 0 &
$\begin{matrix}\longrightarrow\\[-3.5mm]
\longleftarrow\end{matrix}$ & 0 &
$\begin{matrix}\longrightarrow\\[-3.5mm]\longrightarrow\\[-3.5mm]
\longleftarrow\\[-3.5mm]\longleftarrow \end{matrix}$
& $\mathbb Z$\\
\end{tabular}\ \ \ \ \ \ \ \ \ \ \ \ \ \
\begin{tabular}{cccccccccccc}
0 &
$\begin{matrix}\longrightarrow\\[-3.5mm]
\longleftarrow\end{matrix}$ & 0 &
$\begin{matrix}\longrightarrow\\[-3.5mm]\longrightarrow\\[-3.5mm]
\longleftarrow\\[-3.5mm]\longleftarrow \end{matrix}$
& $\mathbb Z$\\
$\downarrow$ & & $\downarrow$ & & $ \downarrow$\\
$\mathbb Z$ &
$\begin{matrix}\buildrel{(3,3)}\over\longrightarrow\\[-2.5mm]
\buildrel{(1,1)}\over\longleftarrow\end{matrix}$ & $\mathbb
Z\oplus \mathbb Z$ & $
\begin{matrix}\buildrel{(0,2)}\over\longrightarrow\\[-2.5mm]\buildrel{(2,0)}\over\longrightarrow\\[-2.5mm]
\buildrel{(0,1)}\over\longleftarrow\\[-2.5mm]\buildrel{(1,0)}\over\longleftarrow \end{matrix}$
& $\mathbb Z$\\
\end{tabular}
\end{center}
This case is very simple due to the structure of cross-effects of
$\Lambda^3$ and can be easily extended to high dimensional
symmetric and exterior powers.

We collect the $Hom$-functors between main quadratic functors in
the category of free abelian groups in the following table:
\vspace{1cm} \begin{center}
\begin{tabular}{ccccccccccccccccc}
 $G\setminus F$ & \vline & $\Gamma_2$ & $\otimes^2$ & $\widetilde\otimes^2$ & $\mathbb Z/2$ & $S^2$ & $\Lambda^2$ & $\Lambda^2\otimes \mathbb Z/2$\\
 \hline
$\Gamma_2$ & \vline & $\mathbb Z$ & $\mathbb Z$ & 0 & 0 & $\mathbb Z$ & 0 & 0\\
$\otimes^2$ & \vline & $\mathbb Z$ & $\mathbb Z\oplus \mathbb Z$ & $\mathbb Z$ & 0 & $\mathbb Z$ & $\mathbb Z$ & 0\\
$\widetilde\otimes^2$ & \vline & $\mathbb Z/2$ & $\mathbb Z$ & $\mathbb Z$ & $\mathbb Z/2$ & 0 & $\mathbb Z$ & 0\\
$\mathbb Z/2$ & \vline & $\mathbb Z/2$ & 0 & 0 & $\mathbb Z/2$ & 0 & 0 & 0\\
$S^2$ & \vline & $\mathbb Z$ & $\mathbb Z$ & 0 & 0 & $\mathbb Z$ & 0 & 0\\
$\Lambda^2$ & \vline & 0 & $\mathbb Z$ & $\mathbb Z$ & 0 & 0 & $\mathbb Z$ & 0\\
$\Lambda^2\otimes \mathbb Z/2$ & \vline & 0 & $\mathbb Z/2$ &
$\mathbb Z/2$ & 0 & $\mathbb Z/2$ & $\mathbb Z/2$ &
$\mathbb Z/2$\\
\end{tabular}
\end{center}
\vspace{.5cm} \begin{center} Table 1. $Hom_f(F,G)$\end{center}

 \vspace{1cm}
Now observe that some of the polynomial functors $F$ of degree $n$
have the property that $Hom(F,G)=Hom(G,F)=0$ for any functor $G$
of degree less than $n$. Let us now consider examples of functors
which do not satisfy this property.

\vspace{.5cm}\noindent 1. For an abelian group $A$, we have a
natural map
$$
\Gamma_2(A)\to A\otimes \mathbb Z/2
$$
The kernel $K(A)$ of the natural map
$$ \Gamma_2(A)\to A\otimes A
$$
defines a functor $K$ in the category of abelian groups (see
\cite{BP} for the description of this functor). A simple analysis
shows that $K$ is a linear functor.

\vspace{.5cm}\noindent 2. There are natural transformations
$$
Tor(A,\mathbb Z/2)\to S^2(A),\ \ a\mapsto a^2,\ a\in A,\ 2a=0
$$
and
$$
A\otimes Tor(A,\mathbb Z/2)\to S^3(A),\ \ a\otimes b\mapsto ab^2,\
a,b\in A,\ 2b=0.
$$

\vspace{.5cm} Now observe, that for any functor $F$ from the set
$\{\Lambda^n,\ \Lambda^n\otimes \mathbb Z/p,\ \otimes^n,\
\otimes^n\otimes \mathbb Z/p\ \ \ (p\ \  \text{is a prime})\},$
the natural map
$$
F(A)\to F(A|\dots|A)\ (n-\text{th cross effect})
$$
induced by the diagonal embedding $A\hookrightarrow A\oplus \dots
\oplus A\ (n\ \text{copies of A})$ is injective. It follows that
$Hom(G,F)=0$ for every functor $G$ of degree less than $n$.
Similarly, the natural projection
$$A\oplus \dots \oplus A\mapsto A,\ (a_1,\dots, a_n)\mapsto
a_1+\dots+a_n,\ a_i\in A,$$ induces a natural epimorphism
$$
F(A|\dots |A)\to F(A)
$$
hence $Hom(F,G)=0$ for every functor $G$ of degree less than $n$.

\subsection{Extensions between functors}
A natural short exact sequence
$$
0\to S^2(A)\otimes \mathbb Z/2\to \Gamma_2(A\otimes \mathbb
Z/4)\to \Gamma_2(A\otimes \mathbb Z/2)\to 0
$$
is given by the following short exact sequence of the quadratic
$\mathbb Z$-modules\footnote{we will always display quadratic
$\mathbb Z$-modules horizontally}:
\begin{equation}\label{mz2}
\xyma{\mathbb Z/2 \ar@{->}[r]^0\ar@{>->}[d] & \mathbb Z/2
\ar@{>->}[d] \ar@{->}[r]^1 & \mathbb Z/2\ar@{>->}[d]\\ \mathbb Z/8
\ar@{->>}[r] \ar@{->>}[d] & \mathbb Z/4 \ar@{>->}[r] \ar@{->>}[d]
& \mathbb Z/8\ar@{->>}[d]\\ \mathbb Z/4\ar@{->>}[r] & \mathbb Z/2
\ar@{>->}[r] & \mathbb Z/4}
\end{equation}
Similarly, the natural exact sequence
$$
0\to \Lambda^2(A)\otimes \mathbb Z/2\to \Gamma_2(A)\otimes \mathbb
Z/2\to A\otimes \mathbb Z/2\to 0
$$
is given by
$$
\xyma{0 \ar@{->}[r]\ar@{->}[d] & \mathbb Z/2 \ar@{->}[d]^1
\ar@{->}[r] & 0\ar@{->}[d]\\ \mathbb Z/2 \ar@{->}[r]^1 \ar@{->}[d]
& \mathbb Z/2 \ar@{->}[r]^0 \ar@{->}[d] & \mathbb Z/2\ar@{->}[d]\\
\mathbb Z/2\ar@{->}[r] & 0\ar@{->}[r] & \mathbb Z/2}
$$

The following proposition follows directly from the structure of
polynomial $\mathbb Z$-module which corresponds to the exterior
power.

\begin{prop}\label{lambda3}
Let $F$ be a functor of degree $d$, then
$Ext_f(F,\Lambda^{d+2})=Ext_f(\Lambda^{d+2},F)=0$.
\end{prop}

Using the language of quadratic $\mathbb Z$-modules, one can
compute the values for $Ext$-functors for main quadratic functors
$F,G: \sf fAb\to Ab$. For example,
\begin{align*}
& Ext_f(\Lambda^2\otimes \mathbb Z/2,\Gamma_2)=\mathbb Z/2\\
& Ext_f(-\otimes \mathbb Z/2, S^2)=\mathbb Z/2\\
& Ext_f(\Lambda^2\otimes \mathbb Z/2, -\otimes\mathbb Z/2)=\mathbb
Z/2\\
& Ext_f(\Gamma_2, S^2)=0\\
& Ext_f(-\otimes \mathbb Z/2, \Lambda^2\otimes \mathbb
Z/2)=\mathbb Z/2\\
& Ext_f(\Gamma_2,\Gamma_2)=0\\
& Ext_f(\Gamma_2,\Lambda^2\otimes \mathbb Z/2)=0
\end{align*}
The proofs are direct, they follow from computations of the
extensions between quadratic $\mathbb Z$-modules which correspond
to the quadratic functors.

The generators of the Ext-groups $Ext_f(\Lambda^2\otimes \mathbb
Z/2,\Gamma_2),$ $Ext_f(-\otimes \mathbb Z/2,S^2),$
$Ext_f(\Lambda^2\otimes\mathbb Z/2, -\otimes \mathbb Z/2)$ one can
find in the following diagram
$$
\xyma{S^2(A)\ar@{=}[r] \ar@{>->}[d] & S^2(A)\ar@{->}[d]^2\\
\Gamma_2(A) \ar@{>->}[r]^g \ar@{->>}[d] & S^2(A)\ar@{->>}[r]
\ar@{->>}[d] & \Lambda^2(A)\otimes \mathbb Z/2\ar@{=}[d]\\
A\otimes \mathbb Z/2 \ar@{>->}[r] & S^2(A)\otimes \mathbb Z/2
\ar@{->>}[r] &\Lambda^2(A)\otimes \mathbb Z/2}
$$
where the map $g$ is given by setting $g: \gamma_2(a)\mapsto a^2,\
a\in A$. The generators of the groups $Ext_f(\Gamma_2\otimes
\mathbb Z/2,-\otimes \mathbb Z/2),$ $Ext_f(-\otimes \mathbb Z/2,
\Lambda^2\otimes \mathbb Z/2)$ on can find in the following
diagram
$$
\xyma{A\otimes \mathbb Z/2 \ar@{>->}[r] \ar@{=}[d] & S^2(A)\otimes
\mathbb Z/2\ar@{->>}[r] \ar@{>->}[d] & \Lambda^2(A)\otimes \mathbb
Z/2\ar@{>->}[d] \\ A\otimes \mathbb Z/2 \ar@{>->}[r] &
\Gamma_2(A\otimes \mathbb Z/2) \ar@{->>}[d] \ar@{->>}[r] &
\Gamma_2(A)\otimes \mathbb Z/2 \ar@{->>}[d] \\ & A\otimes \mathbb
Z/2 \ar@{=}[r] & A\otimes \mathbb Z/2}
$$

\vspace{.5cm}
\section{Homology of abelian groups}
\vspace{.5cm} We will now show that (\ref{seqh3}) does not split
naturally. For an abelian group $A$, recall the bar-resolution
$$
\mathcal B(A):\ \ \ \ \dots\to \mathbb Z[A\oplus A\oplus
A]\buildrel{d_3}\over\to\mathbb Z[A\oplus A]\buildrel{d_2}\over\to
\mathbb Z[A]\buildrel{d_1}\over\to \mathbb Z
$$
where the differential
$$
d_i: \mathbb Z[A^{\oplus i}]\to \mathbb Z[A^{\oplus i-1}]
$$
is given by
$$
d_i: (a_1,\dots, a_i)\mapsto (a_1,\dots,
a_{i-1})+\sum_{j=1}^{i-1}(-1)^j(a_1,\dots, a_j+a_{j+1},\dots,
a_i)+(-1)^i(a_2,\dots, a_i).
$$
There is a natural isomorphism in the derived category $$ \mathbb
Z[A[1]]\simeq \mathcal B(A)
$$ and, in particular, an isomorphism
$$
H_i(A)\simeq H_i(\mathcal B(A)),\ i\geq 0.
$$
Some generators of the homology groups $H_i(A)$ can be easily
described in terms of $\mathcal B(A)$, for example, the map
$$
H_2(A)\simeq \Lambda^2(A)\to ker(d_2)
$$
is given by
$$
a\wedge b\mapsto (a,b)-(b,a),\ a,b\in A.
$$

Consider the functor $H_2(A;\mathbb Z/2)\simeq H_2(\mathcal
B(A)\otimes \mathbb Z/2)$. The universal coefficient theorem
implies that there is a natural exact sequence
\begin{equation}\label{sequ1}
0\to \Lambda^2(A)\otimes \mathbb Z/2 \to H_2(A;\mathbb Z/2)\to
Tor(A,\mathbb Z/2)\to 0
\end{equation}
Consider the map $f:H_2(A;\mathbb Z/2)\to H_2(A| A;\mathbb
Z/2)=A\otimes A\otimes \mathbb Z/2$ induced by the diagonal map
$A\to A\oplus A$. Suppose that the sequence (\ref{sequ1}) splits
naturally, i.e. $H_2(A;\mathbb Z/2)=\Lambda^2(A)\otimes \mathbb
Z/2\oplus Tor(A,\mathbb Z/2).$ Then the composition map
$Tor(A,\mathbb Z/2)\to H_2(A;\mathbb Z/2)\buildrel{f}\over\to
H_2(A|A;\mathbb Z/2)$ is zero, since we have seen that there is no
non-trivial natural transformation between a linear functor and
$A\otimes A\otimes \mathbb Z/2$. In particular, for $A=\mathbb
Z/2$, the map
$$
f: H_2(\mathbb Z/2;\mathbb Z/2)\to H_2(\mathbb Z/2|\mathbb
Z/2;\mathbb Z/2)
$$
is zero. Let $a$ be a generator of $A=\mathbb Z/2$. One has
$H_2(A;\mathbb Z/2)=\mathbb Z/2$ and the generator of this
$\mathbb Z/2$ in the bar-resolution can be chosen as $(a,a)\in
\mathbb Z/2[A\oplus A]$. This follows from the fact that $(a,a)\in
ker(d_2)\setminus im(d_3)$, since $im(d_3)$ lies in the
augmentation ideal of $\mathbb Z/2[A\oplus A]$. Taking
$B=A=\mathbb Z/2$ and $b$ as a generator of $B$, we now consider
the map of bar resolutions $\mathcal B(A)\to \mathcal B(A\oplus
B)$, induced by the diagonal map $A\to A\oplus B$. We see that the
image of the map $f$ is generated by an element $(a+b,a+b)\in
\mathbb Z/2[(A\oplus B)\oplus (A\oplus B)].$ It is easy to verify
that
$$
(a+b,a+b)+(a,a)+(b,b)\equiv (a,b)+(b,a)\mod im(d_3)
$$
in $\mathbb Z/2[(A\oplus B)\oplus (A\oplus B)]$. This implies that
the image of the element $(a+b,a+b)$ under the map $H_2(A\oplus
B;\mathbb Z/2)\to H_2(A|B;\mathbb Z/2)=A\otimes B\otimes \mathbb
Z/2$ is the same as the image of the element $(a,b)+(b,a)$.
However, the image of the element $(a,b)+(b,a)$ in $A\otimes
B\otimes \mathbb Z/2$ is exactly $a\otimes b\otimes 1,$ which is
the generator of $A\otimes B\otimes \mathbb Z/2$. This proves that
the sequence (\ref{sequ1}) does not split functorially.

\begin{lemma}\label{leq1}
There is a natural isomorphism
$$
H_2(A;\mathbb Z/2)\simeq H_3(A|\mathbb Z/2).
$$
\end{lemma}
\begin{proof}
We have
$$
\mathcal B(A\oplus \mathbb Z/2)\simeq \mathcal B(A)\otimes
\mathcal B(\mathbb Z/2)
$$
Since
$$
\mathcal B(\mathbb Z/2)\simeq \mathbb Z\oplus\bigoplus_{n\geq 0}
\mathbb Z/2[2n+1],
$$
we have a natural isomorphism
\begin{multline}\label{za1}
H_3(A\oplus \mathbb Z/2)\simeq H_3(\mathcal B(A)\oplus (\mathcal
B(A)\otimes \mathbb Z/2[1])\oplus (\mathcal B(A)\otimes \mathbb
Z/2[3]))\simeq\\ H_3(A)\oplus H_2(\mathcal B(A)\otimes \mathbb
Z/2)\oplus H_3(\mathbb Z/2),
\end{multline}
where the natural maps $H_3(A)\to H_3(A\oplus \mathbb Z/2)$ and
$H_3(\mathbb Z/2)\to H_3(A\oplus \mathbb Z/2)$ are splitting
monomorphisms on the direct summands in (\ref{za1}). It follows
that
$$
H_3(A|\mathbb Z/2)\simeq H_2(\mathcal B(A)\otimes \mathbb
Z/2)\simeq H_2(A;\mathbb Z/2).
$$
\end{proof}
\begin{cor}\label{corr2}
The natural sequence (\ref{seqh3}) \begin{equation}\label{h3} 0\to
\Lambda^3(A)\to H_3(A)\to \Omega_2(A)\to 0
\end{equation}
does not split functorially.
\end{cor}
\begin{proof}
Suppose that the sequence (\ref{h3}) splits naturally, i.e. there
is a natural isomorphism
$$
H_3(A)\simeq \Lambda^3(A)\oplus \Omega_2(A).
$$
This induces the following natural decomposition for the
cross-effect functor
$$
H_3(A|\mathbb Z/2)\simeq \Lambda^2(A)\otimes \mathbb Z/2\oplus
Tor(A,\mathbb Z/2).
$$
Lemma \ref{leq1} implies that there is a natural isomorphism
$$
H_2(A;\mathbb Z/2)\simeq \Lambda^2(A)\otimes \mathbb Z/2\oplus
Tor(A,\mathbb Z/2),
$$
however this contradicts the fact that the sequence (\ref{sequ1})
does not split functorially.
\end{proof}
Observe that, for any free abelian group, there is a natural
isomorphism
$$ H_2(A\otimes \mathbb Z/2;\mathbb Z/2)\simeq \Gamma_2(A)\otimes
\mathbb Z/2
$$
and the functor $H_2(A\otimes\mathbb Z/2;\mathbb Z/2)$ represents
the non-trivial element of $$ Ext_f(-\otimes \mathbb
Z/2,\Lambda^2\otimes \mathbb Z/2)=\mathbb Z/2.
$$
As a consequence of corollary \ref{corr2}, for a free abelian
group $A$, the functor
$$
A\mapsto H_3(A\otimes \mathbb Z/2)
$$
lives in the following short exact sequence
$$
0\to \Lambda^3(A)\otimes \mathbb Z/2\to H_3(A\otimes \mathbb
Z/2)\to \Gamma_2(A)\otimes \mathbb Z/2\to 0
$$
and is represented by the following cubical $\mathbb Z$-module:
$$
\mathbb Z/2\
\begin{matrix}\buildrel{1}\over\longrightarrow\\[-2.5mm]
\buildrel{0}\over\longleftarrow\end{matrix}\ \mathbb Z/2\
\begin{matrix}\buildrel{1}\over\longrightarrow\\[-2.5mm]\buildrel{1}\over\longrightarrow\\[-2.5mm]
\buildrel{0}\over\longleftarrow\\[-2.5mm]\buildrel{0}\over\longleftarrow
\end{matrix}\
\mathbb Z/2
$$
In particular, this functor represents the non-trivial element in
the corresponding $Ext$-group:
\begin{equation}\label{nonsplitting3} H_3(A\otimes \mathbb Z/2)\neq 0\in
Ext_f(\Gamma_2\otimes \mathbb Z/2,\Lambda^3\otimes \mathbb
Z/2)=\mathbb Z/2.\end{equation}

We are now ready to generalize corollary  \ref{corr2} to the case
of higher homology functors.
\begin{prop}\label{teq}
Let $A$ be an abelian group. For $n\geq 3$, the natural
monomorphism induced by Pontryagin product
$$
\Lambda^n(A)\hookrightarrow H_n(A)
$$
does not split naturally.
\end{prop}
\begin{proof}
For $n=3$ this is corollary \ref{corr2}. Now the result follows by
induction on $n$, observing that there is a natural isomorphism
$$
H_n(A|\ \mathbb Z)\simeq H_{n-1}(A)
$$
which follows from the K\"unneth formula. Indeed, assuming that
the monomorphism $\Lambda^n(A)\hookrightarrow H_n(A)$ splits
naturally, we get the natural splitting of the cross-effects
$$
\xyma{\Lambda^n(A|\ \mathbb Z)\ar@{>->}[r] \ar@{->}[d]^\simeq &
H_n(A|\ \mathbb Z) \ar@{->}[d]^\simeq\\ \Lambda^{n-1}(A)
\ar@{>->}[r] & H_{n-1}(A)}
$$
but the lower map is not split by induction hypothesis.
\end{proof}
\begin{remark}
For an odd prime $p$, the functor
$$
A\mapsto H_3(A\otimes \mathbb Z/p),\ A\in {\sf fAb}
$$
splits as
\begin{equation}\label{spodd}
H_3(A\otimes \mathbb Z/p)=\Lambda^3(A)\otimes \mathbb Z/p\oplus
\Gamma_2(A)\otimes \mathbb Z/p.
\end{equation}
\end{remark}
\begin{proof}
First we prove that $Ext_f(\Gamma_2\otimes \mathbb Z/p,
\Lambda^3\otimes \mathbb Z/p)=0$. Every element of this Ext-group
can be presented as a diagram of the form
\begin{center}
\begin{tabular}{ccccccccccc}
0 & $\begin{matrix}\longrightarrow\\[-2.5mm]\longleftarrow
\end{matrix}$ & 0 & $\begin{matrix}\longrightarrow\\[-3.5mm]\longrightarrow\\[-3.5mm]
\longleftarrow\\[-3.5mm]\longleftarrow \end{matrix}$ & $\mathbb Z/p$\\
$\downarrow$ & & $\downarrow$ & & $\downarrow^=$ \\
$\mathbb Z/p$ & $\begin{matrix}\buildrel{1}\over\longrightarrow\\[-2.5mm]\buildrel{2}\over\longleftarrow
\end{matrix}$ & $\mathbb Z/p$ & $\begin{matrix}\buildrel{h_1}\over\longrightarrow\\[-2.5mm]\buildrel{h_2}\over\longrightarrow\\[-2.5mm]
\buildrel{p_1}\over\longleftarrow\\[-2.5mm]\buildrel{p_2}\over\longleftarrow \end{matrix}$ & $\mathbb Z/p$\\
$\downarrow$ & & $\downarrow$ & & $\downarrow$ \\
$\mathbb Z/p$ &
$\begin{matrix}\buildrel{1}\over\longrightarrow\\[-2.5mm]\buildrel{2}\over\longleftarrow
\end{matrix}$ & $\mathbb Z/p$ & $\begin{matrix}\longrightarrow\\[-3.5mm]\longrightarrow\\[-3.5mm]
\longleftarrow\\[-3.5mm]\longleftarrow \end{matrix}$ & 0
\end{tabular}
\end{center}
It follows immediately that $p_1$ and $p_2$ are zero maps. The
relations
$$
h_1p_1h_1=2h_1,\ h_2p_2h_2=2h_2
$$
imply that $h_1$ and $h_2$ are zero map. Hence
$$Ext_f(\Gamma_2\otimes \mathbb Z/p, \Lambda^3\otimes \mathbb
Z/p)=0.$$ The splitting (\ref{spodd}) follows from the fact that,
for a free abelian $A$, the sequence (\ref{h3}) has the form
$$
0\to \Lambda^3(A)\otimes \mathbb Z/p\to H_3(A\otimes \mathbb
Z/p)\to \Gamma_2(A)\otimes \mathbb Z/p\to 0.
$$
\end{proof}

\vspace{.5cm}
\section{The third stable homotopy group of
$K(A,1)$}\label{thirdsec} \vspace{.5cm}
\subsection{Whitehead's exact sequence.} Let $X$ be a
$(r-1)$-connected CW-complex, $r\geq 2.$ Consider the following
long exact sequence of abelian groups \cite{Wh}:
\begin{equation}\label{white}
\dots H_{n+1}X\to \Gamma_nX\to \pi_nX\buildrel{h_n}\over\to
H_nX\to \Gamma_{n-1}X\to \dots, \end{equation} where
$\Gamma_nX=im\{\pi_nX^{n-1}\to \pi_nX^n\}$ (here $X^i$ is the
$i$-th skeleton of $X$), $h_n$ is the $n$th Hurewicz homomorphism.
The Hurewicz theorem is equivalent to the statement $\Gamma_iX=0,\
i\leq r.$ J.H.C. Whitehead computed the term $\Gamma_{r+1}X$ (see
\cite{Wh})\footnote{Care should be taken to distinguish between
Whitehead's functors $\Gamma_nX$ and the divided power functors
$\Gamma_i(A)$.}:
$$
\Gamma_{r+1}X=\begin{cases} \Gamma_2(\pi_2X),\ r=2\\
\pi_rX\otimes \mathbb Z/2,\ r>2\end{cases}
$$
where $\Gamma_2: \sf Ab\to Ab$ is the universal quadratic functor
(or equivalently the divided square).

Consider the stable analog of the Whitehead exact sequence in low
degrees. Here we recall the description of functors $\Gamma_i,\
i=r+1,r+2,r+3$ from \cite{BG}. Assume that $X$ is
$(r-1)$-connected complex, $r\geq 6$. In this case, we have the
following:
$$
\eta_1: \pi_r(X)\otimes \mathbb Z/2\to \pi_{r+1}(X)
$$
is induced by the Hopf map $\eta_r\in \pi_{r+1}(S^r),$ i.e.
$\eta^1(\alpha\otimes 1)=\alpha\eta_r$ and there is a natural
exact sequence
$$
0\to \pi_{r+1}(X)\otimes \mathbb Z/2\to \Gamma_{r+2}X\to
Tor(\pi_r(X),\mathbb Z/2)\to 0
$$
where the composite map $\pi_{r+1}(X)\otimes \mathbb Z/2\to
\pi_{r+2}(X)$ is induced by the Hopf map
$\eta_{r+1}\in\pi_{r+1}(S^r)$.

The description of the term $\Gamma_{r+3}X$ is given as follows.
There is a natural exact sequence \begin{equation}\label{stranges}
L_2\Gamma^2(\eta_1)\to \Gamma^3(\eta_1,\eta_2)\to \Gamma_{r+3}X\to
L_1\Gamma^2(\eta^1)\to 0,
\end{equation}
where the functors in this sequence can be described in the
following way:
\begin{align*}
& L_1\Gamma^2(\eta_1)=coker\{\pi_r(X)\otimes \mathbb
Z/2\buildrel{\eta_1}\over\to Tor(\pi_{r+1}(X),\mathbb Z/2)\}\\
& L_2\Gamma^2(\eta_1)=ker(\eta_1)
\end{align*}
and $\Gamma^3(\eta_1,\eta_2)=\pi_r(X)\otimes \mathbb Z/3\oplus P,$
where $P$ is given by the pushout
\begin{equation}\label{pushout}
\xyma{\pi_r(X)\otimes \mathbb Z/2 \ar@{->}[d]^{\pi_r(X)\otimes 4}
\ar@{->}[r] &
\pi_{r+2}\otimes \mathbb Z/2\ar@{->}[d]\\
\pi_r(X)\otimes \mathbb Z/8 \ar@{->}[r] & P}
\end{equation}
where the upper horizontal map induced by the map $S^{r+2}\to
S^r$, which defines a generator of $\pi_2^S=\mathbb Z/2$.

\subsection{Spectral sequence} Recall the spectral sequence from \cite{MikhailovWu}. Consider an abelian group $A$ and
its two-step flat resolution
$$
0\to A_1\to A_0\to A\to 0.
$$
By Dold-Kan correspondence, one obtains the following free abelian
simplicial resolution of ~$A$:
$$
N^{-1}(A_1\hookrightarrow A_0):\ \ \ \ \ldots \begin{matrix}\longrightarrow\\[-3.5mm]\longrightarrow\\[-3.5mm]\longrightarrow\\[-3.5mm]\longleftarrow\\[-3.5mm]
\longleftarrow
\end{matrix}\ A_1\oplus s_0(A_0)\ \begin{matrix}\longrightarrow\\[-3.5mm]\longrightarrow\\[-3.5mm]
\longleftarrow \end{matrix}\ A_0.$$

Applying Carlsson construction (see \cite{Carl} or
\cite{MikhailovWu} for the detailed description of this
construction) to the resolution $N^{-1}(A_1\hookrightarrow A_0)$,
we obtain the following bisimplicial group:
\begin{center}
\begin{tabular}{ccccccccccc}
$F^{N^{-1}(A_1\hookrightarrow A_0)_2}(S^n)_3$ & $\begin{matrix}\longrightarrow\\[-3.5mm]\longrightarrow\\[-3.5mm]\longrightarrow\\[-3.5mm]\longrightarrow\\[-3.5mm]\longleftarrow\\[-3.5mm]
\longleftarrow\\[-3.5mm]\longleftarrow
\end{matrix}$ & $F^{N^{-1}(A_1\hookrightarrow A_0)_2}(S^n)_2$ & $\begin{matrix}\longrightarrow\\[-3.5mm] \longrightarrow\\[-3.5mm]\longrightarrow\\[-3.5mm]
\longleftarrow\\[-3.5mm]\longleftarrow \end{matrix}$ & $N^{-1}(A_1\hookrightarrow A_0)_2$\\
$\downarrow\downarrow\downarrow\uparrow\uparrow$ & & $\downarrow\downarrow\downarrow\uparrow\uparrow$ & & $\downarrow\downarrow\downarrow\uparrow\uparrow$ \\
$F^{A_1\oplus s_0(A_0)}(S^n)_3$ & $\begin{matrix}\longrightarrow\\[-3.5mm]\longrightarrow\\[-3.5mm]\longrightarrow\\[-3.5mm]\longrightarrow\\[-3.5mm]\longleftarrow\\[-3.5mm]
\longleftarrow\\[-3.5mm]\longleftarrow
\end{matrix}$ & $F^{A_1\oplus s_0(A_0)}(S^n)_2$ & $\begin{matrix}\longrightarrow\\[-3.5mm] \longrightarrow\\[-3.5mm]\longrightarrow\\[-3.5mm]
\longleftarrow\\[-3.5mm]\longleftarrow \end{matrix}$ & $A_1\oplus s_0(A_0)$\\
$\downarrow\downarrow\uparrow$ & & $\downarrow\downarrow\uparrow$ & & $\downarrow\downarrow\uparrow$ \\
$F^{A_0}(S^n)_3$ & $\begin{matrix}\longrightarrow\\[-3.5mm]\longrightarrow\\[-3.5mm]\longrightarrow\\[-3.5mm]\longrightarrow\\[-3.5mm]\longleftarrow\\[-3.5mm]
\longleftarrow\\[-3.5mm]\longleftarrow
\end{matrix}$ & $F^{A_0}(S^n)_2$ & $\begin{matrix}\longrightarrow\\[-3.5mm] \longrightarrow\\[-3.5mm]\longrightarrow\\[-3.5mm]
\longleftarrow\\[-3.5mm]\longleftarrow \end{matrix}$ & $A_0$
\end{tabular}
\end{center}
Here the $m$th horizontal simplicial group is Carlsson
construction $F^{N^{-1}(A_1\hookrightarrow A_0)_m}(S^n).$ By the
result of Quillen~\cite{Quillen}, we obtain the following spectral
sequence:
\begin{equation}\label{specseq}
E_{p,q}^2=\pi_q(\pi_p\Sigma^n K(N^{-1}(A_1\hookrightarrow
A_0),1))\Longrightarrow \pi_{p+q}\Sigma^n K(A,1).
\end{equation}
In particular, for $n$ sufficiently large, the spectral sequence
becomes
\begin{equation}\label{stabss}
E_{p,q}^2=\pi_q(\pi_p^SK(N^{-1}(A_1\hookrightarrow
A_0),1))\Longrightarrow \pi_{p+q}^SK(A,1).
\end{equation}
where $\pi_n^S$ is the $n$th stable homotopy group.

\subsection{$\pi_3^SK(A,1)$.} We now apply the above results for the description of the stable homotopy groups of $K(A,1)$ in low degrees.
Given an abelian group $A$, the homotopy functors $\pi_*^S: {\sf
Ab\to Ab},\ A\mapsto \pi_*^SK(A,1)$ can be viewed as parts of the
Whitehead exact sequence, which functorially depends on $A$.
Recall that, for $r\geq 2$,
$\pi_2^SK(A,1)=\pi_{r+2}\Sigma^rK(A,1)$ is the antisymmetric
square, and the Whitehead sequence has the form \cite{BL}:
$$
\xyma{\Gamma_{r+2}\Sigma^rK(A,1) \ar@{=}[d] \ar@{>->}[r]&
\pi_{r+2}\Sigma^rK(A,1) \ar@{=}[d]\ar@{->>}[r] & H_2K(A,1)\ar@{=}[d]\\
A\otimes \mathbb Z/2 \ar@{>->}[r]& A\tilde\otimes A\ar@{->>}[r] &
\Lambda^2(A)}
$$

Now consider the next step, the functor
$\pi_3^SK(A,1)=\pi_{r+3}\Sigma^rK(A,1)$ for $r\geq 4.$ First
consider the case of a free abelian group $A$. Observe that, for a
free abelian $A$, one has a natural isomorphism $$A\tilde\otimes
A\otimes \mathbb Z/2\simeq S^2(A)\otimes\mathbb Z/2.
$$
We have the following exact sequence
$$
\xyma{H_4(A)\ar@{=}[d] \ar@{->}[r] & \Gamma_{r+3}\Sigma^rK(A,1)
\ar@{=}[d] \ar@{->}[r] & \pi_3^SK(A,1) \ar@{=}[d] \ar@{->>}[r] & H_3(A)\ar@{=}[d] \\
\Lambda^4(A) \ar@{->}[r] & S^2(A)\otimes \mathbb Z/2 \ar@{->}[r] &
\pi_3^SK(A,1) \ar@{->>}[r] & \Lambda^3(A)}
$$
Now observe that any natural transformation $\Lambda^4(A)\to
S^2(A)\otimes \mathbb Z/2$ is zero, since, for all $n\geq 2$,
there is no non-trivial transformations between $\Lambda^n(A)$ and
any functor of degree less than $n$. Therefore, the functor
$\pi_3^S: {\sf fAb\to Ab}$ lives in the following exact sequence
\begin{equation}\label{3p}
0\to S^2(A)\otimes \mathbb Z/2\to \pi_3^SK(A,1)\to \Lambda^3(A)\to
0
\end{equation}
It follows from a simple analysis of the extensions of the cubical
$\mathbb Z$-modules which correspond to the functors $S^2\otimes
\mathbb Z/2$ and $\Lambda^3$ that any nontrivial extension between
these functors can be given by a diagram of the form
\begin{center}
\begin{tabular}{ccccccccccc}
$\mathbb Z/2$ & $\begin{matrix}\buildrel{0}\over\longrightarrow\\[-2.5mm]\buildrel{1}\over\longleftarrow
\end{matrix}$ & $\mathbb Z/2$ & $\begin{matrix}\longrightarrow\\[-3.5mm]\longrightarrow\\[-3.5mm]
\longleftarrow\\[-3.5mm]\longleftarrow \end{matrix}$ & $0$\\
$\downarrow^=$ & & $\downarrow^=$ & & $\downarrow$ \\
$\mathbb Z/2$ & $\begin{matrix}\buildrel{0}\over\longrightarrow\\[-2.5mm]\buildrel{1}\over\longleftarrow
\end{matrix}$ & $\mathbb Z/2$ & $\begin{matrix}\buildrel{0}\over\longrightarrow\\[-2.5mm]\buildrel{0}\over\longrightarrow\\[-2.5mm]
\buildrel{1}\over\longleftarrow\\[-2.5mm]\buildrel{1}\over\longleftarrow \end{matrix}$ & $\mathbb Z$\\
$\downarrow$ & & $\downarrow$ & & $\downarrow^=$ \\
$0$ &
$\begin{matrix}\longrightarrow\\[-3.5mm]\longleftarrow
\end{matrix}$ & $0$ & $\begin{matrix}\longrightarrow\\[-3.5mm]\longrightarrow\\[-3.5mm]
\longleftarrow\\[-3.5mm]\longleftarrow \end{matrix}$ & $\mathbb Z$
\end{tabular}
\end{center}
and \begin{equation}\label{ext3} Ext_f(\Lambda^3,S^2\otimes
\mathbb Z/2)=\mathbb Z/2.
\end{equation}
We will show now that the extension (\ref{3p}) presents a
non-trivial element of (\ref{ext3}).

\begin{theorem}\label{free3s}
The functor
$$
\pi_3^S: {\sf fAb\to Ab},\ A\mapsto \pi_3^SK(A,1)
$$
is given by the following cubical module:
$$
\mathbb Z/2\
\begin{matrix}\buildrel{0}\over\longrightarrow\\[-2.5mm]
\buildrel{1}\over\longleftarrow\end{matrix}\ \mathbb Z/2\
\begin{matrix}\buildrel{0}\over\longrightarrow\\[-2.5mm]\buildrel{0}\over\longrightarrow\\[-2.5mm]
\buildrel{1}\over\longleftarrow\\[-2.5mm]\buildrel{1}\over\longleftarrow \end{matrix}
\ \mathbb Z
$$
\end{theorem}
\begin{proof}
Assume that, for a free abelian group $A$, the functor
$\pi_3^SK(A,1)$ presents the zero element in (\ref{ext3}), i.e.
$\pi_3^SK(A,1)=S^2(A)\otimes \mathbb Z/2\oplus\Lambda^3(A)$, and
let $B$ be a non-free abelian group. The spectral sequence
(\ref{stabss}) implies that there is a natural exact sequence
\begin{equation}\label{lse3}
0\to S^2(B)\otimes \mathbb Z/2\oplus \Lambda^3(B)\to
\pi_3^SK(B,1)\to L_1\tilde\otimes^2(B)\to 0
\end{equation}
Consider now the functor
$$
\pi_3^SK(-\otimes \mathbb Z/2,1):\ {\sf fAb\to Ab},\ \ A\mapsto
\pi_3^SK(A\otimes \mathbb Z/2,1)
$$
There is the following short exact sequence (see
\cite{MikhailovWu}): $$0\to A\otimes \mathbb Z/2\to
L_1\tilde\otimes^2(A\otimes \mathbb Z/2)\to \Gamma_2(A)\otimes
\mathbb Z/2\to 0$$ and $L_1\tilde\otimes^2(\mathbb Z/2)=\mathbb
Z/4$. Hence $L_1\tilde\otimes^2(A\otimes \mathbb Z/2)$ describes a
nontrivial element of $$Ext(\Gamma_2(A)\otimes \mathbb
Z/2,A\otimes \mathbb Z/2)=\mathbb Z/2$$ and, therefore,
$$
L_1\tilde\otimes^2(A\otimes \mathbb Z/2)\simeq \Gamma_2(A\otimes
\mathbb Z/2).
$$
Therefore, the sequence (\ref{lse3}) can be rewritten for
$B=A\otimes \mathbb Z/2$ as
\begin{equation}\label{qmq} 0\to S^2(A)\otimes \mathbb Z/2\oplus
\Lambda^3(A)\otimes \mathbb Z/2\to \pi_3^SK(A\otimes \mathbb
Z/2,1)\to \Gamma_2(A\otimes \mathbb Z/2)\to 0
\end{equation}
We know from \cite{Liule} that $\pi_3^SK(\mathbb Z/2,1)=\mathbb
Z/8$. The diagram  of cubical $\mathbb Z$-modules which correspond
to the extension (\ref{qmq}) has the following form
\begin{center}
\begin{tabular}{ccccccccccc}
$\mathbb Z/2$ & $\begin{matrix}\buildrel{0}\over\longrightarrow\\[-2.5mm]\buildrel{1}\over\longleftarrow
\end{matrix}$ & $\mathbb Z/2$ & $\begin{matrix}\buildrel{0}\over\longrightarrow\\[-2.5mm]\buildrel{0}\over\longrightarrow\\[-2.5mm]
\buildrel{0}\over\longleftarrow\\[-2.5mm]\buildrel{0}\over\longleftarrow \end{matrix}$ & $\mathbb Z/2$\\
$\downarrow$ & & $\downarrow$ & & $\downarrow^=$ \\
$\mathbb Z/8$ & $\begin{matrix}\buildrel{1}\over\longrightarrow\\[-2.5mm]\buildrel{2}\over\longleftarrow
\end{matrix}$ & $\mathbb Z/4$ & $\begin{matrix}\buildrel{0}\over\longrightarrow\\[-2.5mm]\buildrel{0}\over\longrightarrow\\[-2.5mm]
\buildrel{0}\over\longleftarrow\\[-2.5mm]\buildrel{0}\over\longleftarrow \end{matrix}$ & $\mathbb Z/2$\\
$\downarrow$ & & $\downarrow$ & & $\downarrow$ \\
$\mathbb Z/4$ &
$\begin{matrix}\buildrel{1}\over\longrightarrow\\[-2.5mm]\buildrel{2}\over\longleftarrow
\end{matrix}$ & $\mathbb Z/2$ & $\begin{matrix}\longrightarrow\\[-3.5mm]\longrightarrow\\[-3.5mm]
\longleftarrow\\[-3.5mm]\longleftarrow \end{matrix}$ & $0$
\end{tabular}
\end{center}
One verifies that the above extension is unique and therefore,
$$
\pi_3^S(A\otimes \mathbb Z/2,1)\simeq \Gamma_2(A\otimes \mathbb
Z/4)\oplus \Lambda^3(A)\otimes \mathbb Z/2
$$
(see (\ref{mz2})). Observe also that the Whitehead sequence
implies that the Hurewicz map
$$
\pi_3^SK(A\otimes \mathbb Z/2,1)\to H_3(A\otimes \mathbb Z/2)
$$
is a natural surjection, which induces isomorphism on the triple
cross-effects. However, it is not possible to construct a
commutative diagram of the form
\begin{center}
\begin{tabular}{ccccccccccc}
$\mathbb Z/8$ & $\begin{matrix}\buildrel{1}\over\longrightarrow\\[-2.5mm]\buildrel{2}\over\longleftarrow
\end{matrix}$ & $\mathbb Z/4$ & $\begin{matrix}\buildrel{0}\over\longrightarrow\\[-2.5mm]\buildrel{0}\over\longrightarrow\\[-2.5mm]
\buildrel{0}\over\longleftarrow\\[-2.5mm]\buildrel{0}\over\longleftarrow \end{matrix}$ & $\mathbb Z/2$\\
$\downarrow$ & & $\downarrow$ & & $\downarrow^\simeq$ \\
$\mathbb Z/2$ & $\begin{matrix}\buildrel{1}\over\longrightarrow\\[-2.5mm]\buildrel{0}\over\longleftarrow
\end{matrix}$ & $\mathbb Z/2$ & $\begin{matrix}\buildrel{1}\over\longrightarrow\\[-2.5mm]\buildrel{1}\over\longrightarrow\\[-2.5mm]
\buildrel{0}\over\longleftarrow\\[-2.5mm]\buildrel{0}\over\longleftarrow \end{matrix}$ & $\mathbb Z/2$\\
\end{tabular}
\end{center}
This gives a contradiction. Therefore, the functor $\pi_3^SK(A,1)$
describes a non-trivial element of (\ref{ext3}).
\end{proof}

Theorem \ref{free3s} implies that the functor $\pi_3^SK(-\mathbb
Z/2,1):\ \sf fAb\to Ab$ is represented by the cubical $\mathbb
Z$-module
$$
\mathbb Z/8\
\begin{matrix}\buildrel{1}\over\longrightarrow\\[-2.5mm]
\buildrel{2}\over\longleftarrow\end{matrix}\ \mathbb Z/4\
\begin{matrix}\buildrel{1}\over\longrightarrow\\[-2.5mm]\buildrel{1}\over\longrightarrow\\[-2.5mm]
\buildrel{2}\over\longleftarrow\\[-2.5mm]\buildrel{2}\over\longleftarrow \end{matrix}
\ \mathbb Z/2
$$
The portion of the Whitehead sequence which contains the natural
transformation $\pi_3^S\to H_3$ has the following form
\begin{equation}\label{3dia3}
\xyma{S^2(A)\otimes \mathbb Z/2 \ar@{>->}[r] \ar@{>->}[d] &
F(A)\ar@{>->}[d] \ar@{->>}[r]
& \Lambda^3(A)\otimes \mathbb Z/2\ar@{>->}[d] \\
\Gamma_2(A\otimes \mathbb Z/2) \ar@{>->}[r] \ar@{->>}[d]&
\pi_3^SK(A\otimes
\mathbb Z/2, 1) \ar@{->>}[r] \ar@{->>}[d]& H_3(A\otimes \mathbb Z/2)\ar@{->>}[d]\\
A\otimes \mathbb Z/2 \ar@{>->}[r] & \Gamma_2(A\otimes \mathbb
Z/2)\ar@{->>}[r] & \Gamma_2(A)\otimes \mathbb Z/2} \end{equation}
where the functor $F: \sf fAb\to Ab$ is given by the cubical
$\mathbb Z$-module
$$
\mathbb Z/2\
\begin{matrix}\buildrel{0}\over\longrightarrow\\[-2.5mm]
\buildrel{1}\over\longleftarrow\end{matrix}\ \mathbb Z/2\
\begin{matrix}\buildrel{0}\over\longrightarrow\\[-2.5mm]\buildrel{0}\over\longrightarrow\\[-2.5mm]
\buildrel{1}\over\longleftarrow\\[-2.5mm]\buildrel{1}\over\longleftarrow \end{matrix}
\ \mathbb Z/2
$$
The spectral sequence (\ref{stabss}) therefore implies the
following
\begin{prop}
For an abelian group $A$, there is a natural exact sequence
$$
0\to L_0F(A)\to \pi_3^SK(A,1)\to L_1\tilde\otimes^2(A)\to 0
$$
which does not split.
\end{prop}

\vspace{.5cm}
\section{The fourth stable homotopy group of
$K(A,1)$}\label{fourthsec} \vspace{.5cm} Consider first the case
of a free abelian group $A$. We have the following diagram
$$
\xyma{H_5(A)\ar@{=}[d] \ar@{->}[r] & \Gamma_{r+4}\Sigma^rK(A,1)
\ar@{=}[d] \ar@{->}[r] & \pi_4^SK(A,1) \ar@{=}[d] \ar@{->>}[r] & H_4(A)\ar@{=}[d] \\
\Lambda^5(A) \ar@{->}[r] & A\otimes \mathbb Z/3\oplus P
\ar@{->}[r] & \pi_4^SK(A,1) \ar@{->>}[r] & \Lambda^4(A)}
$$
where the term $P$ was defined in (\ref{pushout}). The functor
$A\otimes \mathbb Z/3\oplus P$ is cubical. The natural
transformation $\Lambda^5(A)\to A\otimes \mathbb Z/3\oplus P$ is
zero and we have a natural short exact sequence
\begin{equation}\label{waw2}
0\to A\otimes \mathbb Z/3\oplus P\to \pi_4^SK(A,1)\to
\Lambda^4(A)\to 0.
\end{equation}
Here, by (\ref{pushout}), the functor $P$ is given by the pushout
diagram
$$
\xyma{A\otimes \mathbb Z/2 \ar@{->}[d]^{A\otimes 4} \ar@{->}[r] &
\pi_3^SK(A,1)\ar@{->}[d]\\ A\otimes \mathbb Z/8 \ar@{->}[r] & P}
$$
It follows that the functor $P$ can be descibed by the cubical
$\mathbb Z$-module:
$$
\mathbb Z/8\
\begin{matrix}\buildrel{0}\over\longrightarrow\\[-2.5mm]
\buildrel{4}\over\longleftarrow\end{matrix}\ \mathbb Z/2\
\begin{matrix}\buildrel{0}\over\longrightarrow\\[-2.5mm]\buildrel{0}\over\longrightarrow\\[-2.5mm]
\buildrel{1}\over\longleftarrow\\[-2.5mm]\buildrel{1}\over\longleftarrow \end{matrix}
\ \mathbb Z/2
$$

\begin{theorem}\label{free4s} The functor
$$
\pi_4^S: {\sf fAb\to Ab},\ A\mapsto \pi_4^SK(A,1)
$$
is described by the following quartic $\mathbb Z$-module
\begin{equation}\label{vm4}
\mathbb Z/24\
\begin{matrix}\buildrel{0}\over\longrightarrow\\[-2.5mm]
\buildrel{12}\over\longleftarrow\end{matrix}\ \mathbb Z/2\
\begin{matrix}\buildrel{0}\over\longrightarrow\\[-2.5mm]\buildrel{0}\over\longrightarrow\\[-2.5mm]
\buildrel{1}\over\longleftarrow\\[-2.5mm]\buildrel{1}\over\longleftarrow \end{matrix}
\ \mathbb Z/2\
\begin{matrix}\buildrel{0}\over\longrightarrow\\[-2.5mm]\buildrel{0}\over\longrightarrow\\[-2.5mm]\buildrel{0}\over\longrightarrow\\[-2.5mm]
\buildrel{1}\over\longleftarrow\\[-2.5mm]
\buildrel{1}\over\longleftarrow\\[-2.5mm]\buildrel{1}\over\longleftarrow
\end{matrix}\ \mathbb Z
\end{equation}
\end{theorem}
\begin{proof}
The proof is similar to that of theorem \ref{free3s}. Since the
quartic $\mathbb Z$-module which corresponds to the functor
$\Lambda^4$ has a simple form, it is easy to see that
\begin{equation}\label{ext5}
Ext_f(\Lambda^4,-\otimes \mathbb Z/3\oplus P)=\mathbb Z/2
\end{equation}
and a nontrivial element in (\ref{ext5}) is given by the quartic
$\mathbb Z$-module (\ref{vm4}). It remains to show that the
sequence (\ref{waw2}) does not split naturally.

Let us assume that the sequence (\ref{waw2}) does split. Recall
that
$$Hom(\Lambda^4\otimes \mathbb Z/2, G)=Hom(G,\Lambda^4\otimes
\mathbb Z/2)=0$$ for every cubical functor $G$. The spectral
sequence (\ref{stabss}) implies that the functor
$$
\pi_4^S(-\otimes\mathbb Z/2,1):\ {\sf fAb\to Ab},\ \ A\mapsto
\pi_4^S(A\otimes \mathbb Z/2,1)
$$
can be represented as a direct sum
$$
\pi_4^S(A\otimes \mathbb Z/2,1)\simeq Q\oplus \Lambda^4(A)\otimes
\mathbb Z/2
$$
for some cubical functor $Q: \sf fAb\to Ab$. The description
(\ref{stranges}) of $\Gamma_{r+4}\Sigma^4K(A\otimes \mathbb
Z/2,1)$ implies that it is a cubical functor for all $r$. It
follows that the image of the Hurewicz map
$$\pi_4^SK(A\otimes \mathbb Z/2,1)\to H_4(A\otimes \mathbb Z/2)$$
also has the form $\bar Q\oplus \Lambda^4(A)\otimes \mathbb Z/2$
for some cubical functor $\bar Q$. The diagram (\ref{3dia3})
implies that the Hurewicz map is an epimorphism, hence we obtain
that there is a natural isomorphism
$$
H_4(A\otimes \mathbb Z/2)\simeq \bar Q\oplus \Lambda^4(A)\otimes
\mathbb Z/2.
$$
Since we know that there exists a natural exact sequence (see
\cite{Breen})
\begin{equation}\label{h4s}
0 \to \Lambda^4(A)\otimes \mathbb Z/2\to H_4(A\otimes \mathbb
Z/2)\to L_1\Lambda^3(A\otimes \mathbb Z/2)\to 0,
\end{equation}
we have an isomorphism $\bar Q\simeq L_1\Lambda^3(A\otimes \mathbb
Z/2)$ and the sequence (\ref{h4s}) splits. To prove this
rigorously, one considers a quartic $\mathbb Z$-module associated
to $\bar Q\oplus \Lambda^4(A)\otimes \mathbb Z/2$ and compares it
with any possible extension of type (\ref{h4s}). It remains to
observe that
$$
H_4(A\otimes \mathbb Z/2|\ \mathbb Z)\simeq H_3(A\otimes \mathbb
Z/2)
$$
so that splitting of (\ref{h4s}) implies the splitting of
$$
0\to \Lambda^3(A)\otimes \mathbb Z/2\to H_3(A\otimes \mathbb
Z/2)\to \Gamma_2(A)\otimes \mathbb Z/2\to 0.
$$
However, by (\ref{nonsplitting3}), this is not possible.
\end{proof}
\vspace{.5cm}
\section{The Splitting of the derived functors}
\vspace{.5cm}
\subsection{Derived functors.}
Let $A$ be an abelian group, and $F$  an endofunctor on the
category of abelian groups. Recall that for every  $n\geq 0$ the
derived functor of $F$
 in the sense
of Dold-Puppe  \cite{DoldPuppe} are defined by
$$
L_iF(A,n)=\pi_i(FKP_\ast[n]),\ i\geq 0
$$
where $P_\ast \to A$ is a projective resolution of $A$, and
 $K$ is  the Dold-Kan transform,  the inverse to the Moore normalization  functor
\[
N:  \mathrm{Simpl}({\sf Ab}) \to \mathrm{Chain}({\sf Ab})
\]
from simplicial abelian groups to chain complexes. We denote by
$LF(A,n)$ the object in the homotopy category of simplicial
abelian groups determined by the simplicial abelian group
$FK(P_\ast[n])$, so that
\[ L_iF(A,n) = \pi_i(LF(A,n))\,.\]
  We
set $LF(A) :=LF(A,0)$ and  $L_iF(A):= L_iF(A,0)$ for any $\ i\geq
0$.

\subsection{The splitting of the derived functors of $\Gamma_2$.} The natural exact sequence
$$
0\to S^2(A)\to \Gamma_2(A)\to A\otimes \mathbb Z/2\to 0
$$
implies that, for an object $C\in \sf D\sf Ab_{\leq 0}$, one has a
distinguished triangle
\begin{equation}\label{tri}
LS^2(C)\to L\Gamma_2(C)\to C\lotimes \mathbb Z/2\to LSP^2(C)[1]
\end{equation}

\begin{theorem}\label{the1}
For any $C\in \sf D\sf Ab_{\leq 0},$ there are natural
isomorphisms
\begin{equation}\label{iso1} \pi_i(L\Gamma_2(C[1]))\simeq
\pi_i(LS^2(C[1]))\oplus \pi_i\left(C\lotimes \mathbb Z/2[1]\right)
\end{equation}
for all $i\geq 0$.
\end{theorem}

The following lemma follows from (Satz 12.1 \cite{DoldPuppe}).
\begin{lemma}\label{lem1}
Let $C\in \sf DAb_{\leq 0}$ be such that $H_0(C)=0,$ then one has
$\pi_1(S^2(C))=0.$ If $H_i(C)=0$ for $i\leq m\ (m\geq 1)$, then
$$
\pi_i(LS^2(C))=0,\ \ i\leq m+2
$$
\end{lemma}

\begin{lemma}\label{lem2}
For every $C\in \sf DAb_{\leq 0},$ the suspension homomorphism
$$
\pi_1(LS^2(C))\to \pi_2(LS^2(C[1]))
$$
is the zero map.
\end{lemma}
\begin{proof}
We have the following natural diagram
\begin{equation}\label{diagm}
\xyma{\pi_1(LS^2(C))\ar@{->}[r]^\simeq \ar@{->}[d]^{\text{susp}} &
L_1S^2(H_0(C))\ar@{->}[d] \\ \pi_2(LS^2(C[1])) \ar@{->}[r]^\simeq
& \Lambda^2(H_0(C)) } \end{equation} The right-hand vertical map
is zero by (Corollary 6.6, \cite{DoldPuppe}). Another way to see
why this map is trivial is to write the cross-effect spectral
sequence for $\pi_*(LS^2(C[1])$ from \cite{DoldPuppe}. The first
page of this spectral sequence implies that there is an exact
sequence
\begin{multline*}
0\to L_1\Lambda^2(H_0(C))\to Tor(H_0(C),H_0(C))\to L_1S^2(H_0(C))\to\\
\Lambda^2(H_0(C))\to H_0(C)\otimes H_0(C)\to S^2(H_0(C))\to 0
\end{multline*} where the middle map is the map from
(\ref{diagm}). It is zero map since the natural transformation
$\Lambda^2(H_0(C))\to H_0(C)\otimes H_0(C)$ is injective.
\end{proof}
\vspace{.5cm}\noindent{\it Proof of theorem \ref{the1}} The proof
is by induction on $i$. Lemma \ref{lem1} implies that there is a
natural isomorphism
$$
\pi_1(L\Gamma_2(C[1]))\simeq \pi_1\left(C\lotimes \mathbb
Z/2[1]\right)
$$
which is induced by the map $L\Gamma_2(C[1])\to C\lotimes \mathbb
Z/2[1]$ from (\ref{tri}).

Let us consider separately the case $i=2$. The assertion follows
from the suspension diagram
$$
\xyma{\pi_3\left(C\lotimes\mathbb Z/2[1]\right)\ar@{=}[d]
\ar@{->}[r] & \pi_2(LS^2(C[1]))\ar@{->}[r] &
\pi_2(L\Gamma_2(C[1]))
\ar@{->>}[r] & \pi_2\left(C\lotimes \mathbb Z/2[1]\right) \ar@{=}[d]\\
\pi_2\left(C\lotimes \mathbb Z/2\right) \ar@{->}[r] &
\pi_1(LS^2(C))\ar@{->}[u]^0 \ar@{->}[r] & \pi_1(L\Gamma_2(C))
\ar@{->}[u] \ar@{->>}[r] & \pi_1\left(C\lotimes \mathbb
Z/p\right)\ar@{->}[lu]}
$$
where the left hand vertical homomorphism is zero by lemma
\ref{lem2}.

Now assume by induction, that, for some $j\geq 2$ and for all
$i\geq j$, there are natural isomorphisms (\ref{iso1}), induced by
(\ref{tri}). Representing the object $C$ as $$\dots \to
C_i\buildrel{\partial_i}\over\to C_{i+1}\to \dots,$$ consider the
subcomplex $Z$ defined by
\begin{align*}
& Z_i=C_i,\ i\geq j-1,\\ & Z_{j-2}=im(\partial_{i-1}),\\
& Z_i=0,\ i<j-2
\end{align*}
The complex $Z$ has the following properties:\\ \\
1) the natural map $Z\to C$ induces isomorphisms
$$
\pi_i\left(Z\lotimes \mathbb Z/2\right)\simeq \pi_i\left(C\lotimes
\mathbb Z/2\right),\ i\geq j;
$$
2) $H_i(Z)=0,\ i\leq j-2.$

\vspace{.25cm}

Consider the natural diagram \begin{equation}\label{dia2}{\small
\xyma{\pi_{j+2}\left(C\lotimes \mathbb
Z/2[1]\right)\ar@{->}[r]\ar@{=}[d] &
\pi_{j+1}(LS^2(C[1]))\ar@{->}[r] & \pi_{j+1}(L\Gamma_2(C[1]))
\ar@{->>}[r] & \pi_{j+1}\left(C\lotimes \mathbb Z/2[1]\right)\\
\pi_{j+2}\left(Z\lotimes \mathbb Z/2[1]\right) \ar@{->}[r] &
\pi_{j+1}(LS^2(Z[1]))\ar@{->}[r] \ar@{->}[u]&
\pi_{j+1}(L\Gamma_2(Z[1])) \ar@{->>}[r]\ar@{->}[u] &
\pi_{j+1}\left(Z\lotimes \mathbb Z/2[1]\right)\ar@{=}[u]}}
\end{equation}
Lemma \ref{lem1} implies that $\pi_{j+1}(LS^2(Z[1]))=0$. The
required splitting now follows from diagram (\ref{dia2}). The
inductive step is complete so that the splitting (\ref{iso1})
proved for all $i$.\ $\Box$

\begin{prop}
The sequence
\begin{equation}\label{oooo}
LS^2(C[1])\to L\Gamma_2(C[1])\to C\lotimes \mathbb Z/2[1]
\end{equation}
does not split in the category $\sf DAb_{\leq 0}$.
\end{prop}
\begin{proof}
We will prove the statement for the simplest case, when $C$ is a
free abelian group. Suppose that $L\Gamma_2(C[1])\simeq
LS^2(C[1])\oplus C\lotimes \mathbb Z/2[1]$. Then
\begin{multline*}
\pi_2\left(L\Gamma_2(C[1])\lotimes \mathbb Z/2\right)\simeq
\pi_2\left(LS^2(C[1])\lotimes \mathbb Z/2\oplus C\lotimes\mathbb
Z/2\lotimes \mathbb Z/2[1]\right)\simeq\\ \Lambda^2(C)\otimes
\mathbb Z/2\oplus C\otimes \mathbb Z/2
\end{multline*}
However, $L\Gamma_2(C[1])$ can be represented by complex
$$
(C\otimes C\otimes \mathbb Z/2\to \Gamma_2(C)\otimes\mathbb
Z/2)[1]
$$
with the obvious map, and
$$
\pi_2\left(L\Gamma_2(C[1])\lotimes \mathbb
Z/2\right)=\ker\{C\otimes C\otimes \mathbb Z/2\to
\Gamma_2(C)\otimes\mathbb Z/2\}.
$$
However, there are no non-trivial natural transformation $C\otimes
\mathbb Z/2\to C\otimes C\otimes \mathbb Z/2$. This contradicts to
the splitting of (\ref{oooo}).
\end{proof}

\vspace{.5cm} \noindent{\it Acknowledgement.} The author thanks L.
Breen for various discussions related to the subject of this paper
and C. Vespa for important suggestions and corrections of some
computations.

\end{document}